\documentclass[12pt]{article}

\usepackage{latexsym,amsmath,amsfonts,amscd,amsthm,upref,graphics,amssymb}
\usepackage{graphicx}
\usepackage{mathrsfs}
\usepackage[all]{xy}
\newtheorem{Theorem}{Theorem}[section]
\newtheorem{lemma}{Lemma}[section]

\newtheorem{remark}{Remark}[]

\newtheorem{cor}{Corollary}[section]

\title{Asymptotic behavior  and moderate deviation
principle for the maximum of a Dyck path }

\author
{Termeh Kousha\\
\footnotesize{\emph{School of Mathematics and Statistics, University
of Ottawa}}\\
 \footnotesize{\emph{  585 King Edward Ottawa, ON K1N
6N5}}\\
 \footnotesize{\emph{e-mail}: \texttt{tkous045@uottawa.ca}}}
\date{}
\begin{document}

\maketitle

\begin{abstract}
In this paper, we obtain a large and  moderate deviation principle
for the law of the maximum of a random Dyck path. Our result extends
the results of Chung \cite{chung}, Kennedy \cite{kennedy} and
Khorunzhiy and Marckert \cite{khor}.
\end{abstract}

\section{Introduction.}
Let $W(t),\ 0\leq t < \infty$ denote a standard Brownian motion and
define ~\\$\tau_{1}=\sup\{t<1:W(t)=0\}\  \textrm{and} \
\tau_{2}=\inf\{t>1:W(t)=0\}.$ Define the process $W_{1}(s),\ 0\leq s
\leq 1$ by setting
$$W_{1}(s)=\frac{\left|W(s\tau_{2}+(1-s)\tau_{1})\right|}{(\tau_{2}-\tau_{1})^{\frac{1}{2}}}.$$
The process $W_{1}$ is known as the \emph{unsigned, scaled Brownian
excursion process} \cite{kennedy}. Chung \cite{chung} and Kennedy
\cite{kennedy} derived the distribution of the maximum of the
unsigned scaled Brownian
excursion.\\
For a path in the lattice $\mathbb{Z}^2$ the NE-SE path is a path
which starts at $(0,0)$ and makes steps either of the form $(1,1)$
(North-East steps) or of the form $(1,-1)$ (South-East steps). A
Dyck path is a NE-SE path which ends on the $x-$axis and never goes
below the $x-$axis. It is known that after normalization a Dyck path
converges to $W_1$ in distribution \cite{kaigh}. For any  integer
$N>0$, define the set of Dyck paths with length of $2N$ as follows,
$$D_{2N}=\{S:=(S_i)_{0\leq i \leq 2N}: S_0=S_{2N}=0, S_{i+1}=S_{i}\pm
1, S_i\geq 0 \ \forall \ i\in [0,2N-1]\}.$$ Khorunzhiy and Marckert
\cite{khor} showed that  for any $\lambda>0$,
$$E\left(\exp(\lambda(2N^{-\frac{1}{2}}\max_{1\leq i \leq 2N} S_i))\right)$$
converges and coincides with the moment generating functions
of the normalized Brownian excursion on $[0,1]$.\\
In this paper, by using other methods, relying on the spectral
properties of an associated  adjacency matrix,  we find the
distribution of the maximum of a Dyck path and show that it has the
same distribution function as the unsigned Brownian excursion which
was first derived in 1976 by Kennedy \cite{kennedy}.  We also obtain
large and moderate deviation principles for the law of the maximum
of a random Dyck path.\\
The paper is organized as follows.  In Section 2 we start by giving
some notation and finding a representation of the Catalan number.
Using this representation we find the distribution of the maximum of
the Dyck path for the case where the length of the Dyck path is
proportional to the square root of the height. This result is
already known but we reprove it by a different approach. In Section
3, we consider different cases and we find moderate and large
deviation principles for the law of the maximum of a random Dyck
path for those cases.

\section{Notation and some results}
For any  integer $N\geq 0$, we  denote by $C_{N}$ the $N$th Catalan
number,
$$C_{N}:=\frac{1}{N+1}\left(\begin{array}{c}
                                                      2N \\
                                                      N
                                                    \end{array}\right)=\frac{
                                                    (2N)!}{\left(N\right)!(N+1)!}.$$
It is well known that $C_{N}=|D_{2N}|$ where $|x|$ gives the
cardinality of $x$.\\
From Stirling's formula, as $N$ goes to infinity the Catalan number
satisfies
$C_{N}=\frac{4^{N}}{(N)^{\frac{3}{2}}\sqrt{\pi}}(1+o(1)).$\\
Let $D_{2N,n}=\{S:S\in D_{2N},\max_{1\leq i \leq 2N}S_i < n\}$. Note
that if $N<n$ then $D_{2N,n}=D_{2N}$.
\begin{lemma}\label{1} For positive integers $n$ and $N$, we have
\begin{eqnarray}
|D_{2N,n}|
=\sum_{s=1}^{n}\left(\frac{2}{n+1}\right)\sin^{2}\left(\frac{\pi
s}{n+1}\right)\left(2\cos\left(\frac{\pi s
}{n+1}\right)\right)^{2N}.
\end{eqnarray}
\end{lemma}
See \cite[p.329]{analytic}  for the proof. An alternative proof
below is  provided relying on the spectral properties of an
associated adjacency matrix , cf also \cite[p.25]{hiai}. As a
corollary, we obtain a formula for the Catalan number.
\begin{proof}[Sketch of the proof]
We consider the graph with $n$ vertices $\{1,\ldots ,n\}$ and $n-1$
edges $\{\{1,2\},\{2,3\},\ldots \{n-1,n\}\}$. This is the graph of a
tree (and actually, of a line). Let $T=(t_{ij})_{n\times n}$ be the
adjacency matrix of this graph, i.e.,
$t_{ij}=\delta_{i,j+1}+\delta_{i,j-1}$ where $\delta_{i,j}$ is the
Kronecker delta. We observe that,
$$|D_{2N,n}|=(T^{2N})_{11}.$$
This follows directly from the definition of $D_{2N,n}$ and from properties of adjacency graphs.
Alternatively, we can write
$$|D_{2N,n}|=v^{T}T^{2N}v,$$
where $v=(1,0,\ldots ,0)^{T}$. According to J. L. Lagrange's
computation in 1759 \cite{hiai}, we have
$$T=UDU^{T},$$
where
$$D=\textrm{diag} \left(2\cos\left(\frac{\pi
k}{n+1}\right)\right)_{k=1}^{n},$$ and the entries of the unitary
matrix $U$ satisfy
$$U_{kl}=\sin\left(\frac{kl\pi
}{n+1}\right)\sqrt{\frac{2}{n+1}}.$$ Putting these facts together
and performing the matrix multiplication yields the desired
conclusion.
\end{proof}
\begin{cor}For positive integers $n>0$ and  $0<N\leq n-1$, the
$N$th Catalan number satisfies $C_{N}=|D_{2N}|=|D_{2N,n}|$.
Therefore,
$$C_{N}=
\sum_{s=1}^{n}\left(\frac{2}{n+1}\right)\sin^{2}\left(\frac{\pi
s}{n+1}\right)\left(2\cos\left(\frac{\pi s
}{n+1}\right)\right)^{2N}.$$
\end{cor}
If $\mathbb{P}_N$ is the  uniform distribution on $ D_{2N}$, then
$$\frac{|D_{2N,n}|}{C_{N}}=\mathbb{P}_N( \textrm{max height of the Dyck
paths in}\ 2N \ \textrm{steps} \ < n).$$ We recall the following
result. It follows from elementary calculus, so we omit its proof.
\begin{lemma}\label{cos}
For all $0\leq x \leq \frac{\pi}{2}$, we have $$\cos(x)\leq
\exp\left(-\frac{x^{2}}{2}\right).$$
\end{lemma}
For  positive integers $n,N$ and $s$ define,
$$G_{N,n}(s)=\sin^2\left(\frac{\pi
s}{n+1}\right)\cos^{2N}\left(\frac{\pi s}{n+1}\right).$$ For a fixed
$n$ we have
$\sin\left(\frac{\pi}{n+1}\right)=\sin\left(\frac{n\pi}{n+1}\right)$
and
$\cos\left(\frac{\pi}{n+1}\right)=-\cos\left(\frac{n\pi}{n+1}\right)$.
So by symmetry we have $G_{N,n}(1)=G_{N,n}(n)$ and for even $n$,
$\sum_{s=1}^{n}G_{N,n}(s)=2\sum_{s=1}^{\frac{n}{2}}G_{N,n}(s).$ Note
that, without lost of generality for large value of $n$, we can
assume $n$ is even. Since if $n$ is odd we have the sum is equal to
$2\sum_{s=1}^{\frac{n-1}{2}}G_{N,n}(s)+ \ \sin^{2}\left(\frac{n\pi
}{n+1}\right)\left(2\cos\left(\frac{n\pi }{n+1}\right)\right)^{2N}$
and as $n\rightarrow\infty$ the last
term will go to zero.\\
Let $[x]$ be the largest integer less than or equal to x. The
following result can be obtained from Fibonacci-Chebyshev
polynomials \cite[p.329]{analytic} and applying Lagrange-B\"{u}rmann
inversion theorem \cite[p.732]{analytic}. Below we give  an
alternative proof in the spirit of this paper using a direct
asymptotic analytic approach.
\begin{Theorem}\label{kappa2}
Let $N=[tn^2]$ where $t$ is  any positive number. We have
$$\lim_{n\rightarrow \infty}
\frac{|D_{2N,n}|}{C_{N}}=f(t),$$ where
$f(t)=4\sqrt{\pi}t^{\frac{3}{2}}\sum_{s=1}^{\infty}s^2\pi^2\exp\left(-{ts^2\pi^2}\right).$
\end{Theorem}
\begin{proof}
We have by symmetry,
\begin{eqnarray*}
\frac{|D_{2N,n}|}{C_{N}}&= &\frac{2\frac{
4^{N}}{(n+1)^3}\sum_{s=1}^{n}(n+1)^2 G_{N,n}(s) }
{\frac{4^{N}}{\sqrt{\pi}N^{\frac{3}{2}}}}\\
&=&\frac{\frac{ 4^{N+1}}{(n+1)^3}\sum_{s=1}^{\frac{n}{2}}(n+1)^2
G_{N,n}(s) } {\frac{4^{N}}{\sqrt{\pi}N^{\frac{3}{2}}}}
\end{eqnarray*}
We split the sum into two cases where $ 1\leq s \leq
\frac{\sqrt{n+1}}{(\ln n)^{\frac{1}{4}}}$ and
$\frac{\sqrt{n+1}}{(\ln n)^{\frac{1}{4}}}\leq s\leq \frac{n}{2}$. By
Lemma \ref{cos}, for all $1\leq s\leq \frac{n}{2}$, we have
$$\cos\left(\frac{\pi s}{n+1}\right)\leq
\exp\left(-\frac{\pi^{2}s^2}{2(n+1)^2}\right).$$ For all $
\frac{\sqrt{n+1}}{(\ln n)^{\frac{1}{4}}}\leq s\leq \frac{n}{2}$, we
have
$$G_{tn^2,n}(s)\leq \exp\left(\frac{-\pi^2(n+1)}{\sqrt{\ln
n}}\right).$$ Since $\cos\left(\frac{\pi s }{n+1}\right)$ decreases
on this interval. Thus
\begin{eqnarray*}
\lim_{n \rightarrow \infty}\ (n+1)^2\sum_{s=\frac{\sqrt{n+1}}{(\ln
n)^{\frac{1}{4}}}}^{\frac{n}{2}}G_{tn^2,n}(s)&\leq&
\lim_{n\rightarrow\infty}(n+1)^2\frac{n}{2}
\exp\left(\frac{-\pi^2(n+1)}{\sqrt{\ln n}}\right)\\&\rightarrow& 0.
\end{eqnarray*}
Let $g(s)=\exp\left(-t\pi^2s^2\right)$, then $g$ is a real valued
function and $\sum_{s=1}^{\infty}g(s)<\infty$. So we have
$$\sum_{s=1}^{n}\cos\left(\frac{\pi s
}{n+1}\right)^{2tn^2}\leq\sum_{s=1}^{n}\exp\left(\frac{-t\pi^2s^2n^2}{(n+1)^2}\right)<\infty.$$
Moreover, we have  for all $1\leq s \leq\frac{\sqrt{n+1}}{(\ln
n)^{\frac{1}{4}}}$ fixed,
$$\lim_{n\rightarrow\infty}\cos\left(\frac{\pi s
}{n+1}\right)^{2tn^2}\rightarrow g(s)\ \ \ \textrm{and,}$$
$$\lim_{n\rightarrow\infty}(n+1)^2\sin^2\left(\frac{\pi s
}{n+1}\right)\rightarrow \pi^2 s^2.$$ Hence for all $ 1\leq s\leq
\frac{\sqrt{n+1}}{(\ln n)^{\frac{1}{4}}}$,
$(n+1)^2G_{tn^2,n}(s)$ converges pointwise to $s^2\pi^2g(s)$.\\
We also have,
$$n^{2}\sin^{2}\left(\frac{\pi s
}{n+1}\right)\leq n^2 \frac{\pi^{2} s^2 }{(n+1)^2}<\pi^2s^2.$$ So
$n^{2}\sin^{2}\left(\frac{\pi s }{n+1}\right)\cos\left(\frac{\pi s
}{n+1}\right)^{2tn^2}< s^2\pi^2g(s),$ where
$\sum_{s=1}^{\infty}s^2\pi^2g(s)<\frac{1}{4t^{\frac{3}{2}}\sqrt{\pi}}$.
So by Dominated Convergence Theorem we have,
$$\lim_{n\rightarrow\infty}\sum_{s=1}^{\frac{\sqrt{n+1}}{(\ln n)^{\frac{1}{4}}}}(n+1)^2
G_{tn^2,n}(s)\rightarrow\sum_{s=1}^{\infty}s^2\pi^2g(s).$$
\begin{eqnarray*}
\lim_{n\rightarrow\infty}\frac{|D_{2tn^2,n}|}{C_{N}}
&=&\lim_{n\rightarrow\infty}\frac{\frac{4^{tn^2+1}}{(n+1)^3}
\sum_{s=1}^{\frac{\sqrt{n+1}}{(\ln n)^{\frac{1}{4}}}}(n+1)^2
G_{tn^2,n}(s)}{\frac{4^{tn^{2}}}{\sqrt{\pi}t^{\frac{3}{2}}n^3}}\\
&+ &  \lim_{n\rightarrow\infty}
\frac{4t^{\frac{3}{2}}\sqrt{\pi}n^3}{{(n+1)^3}}\sum_{\frac{\sqrt{n+1}}{(\ln
n)^{\frac{1}{4}}}}^{\frac{n}{2}} (n+1)^2 G_{tn^2,n}(s) \\
&=&\lim_{n\rightarrow\infty}\frac{4\left[4^{tn^2}\left(\frac{1}{(n+1)^3}\right)\right]}
{\frac{4^{tn^{2}}}{\sqrt{\pi}t^{\frac{3}{2}}n^3}}\sum_{s=1}^{\infty}s^2\pi^2
\exp \left(-ts^2\pi^2\right)+0\\
&=&
4\sqrt{\pi}t^{\frac{3}{2}}\sum_{s=1}^{\infty}s^2\pi^2\exp\left(-{ts^2\pi^2}\right).
\end{eqnarray*}
\end{proof}
Let  $x\sqrt{2N}=n$,  that is  $x^2=\frac{1}{2t}$ or equivalently
$t=\frac{1}{2x^2}$. We can reformulate our result as:
$$f(x):=\lim_{N\rightarrow \infty}\mathbb{P}_{N}(\max_{0\leq i\leq n} S_i \leq x\sqrt{2N}
)=4\sqrt{\pi}(2)^{-\frac{3}{2}}x^{-3}\sum_{s=1}^{\infty}s^2\pi^2\exp\left(-\frac{{\pi^2s^2}}{2x^2}\right).$$
Let $K$ denote the function used by Kennedy \cite{kennedy} and Chung
\cite{chung}  as the distribution of the maximum of the unsigned
scaled Brownian excursion,
$$K(x):=\mathbb{P}\left(\max_{s \in [0,1]} W_1(s)\leq
x\right)=1-2\sum_{s=1}^{\infty}\left(4x^2s^2-1\right)\exp\left(-2x^2s^2\right),\
\  \textrm{for}\  x>0.$$ Our Theorem \ref{kappa2} together with
\cite{kennedy}, \cite{chung},  and \cite{khor} implies the
following. Note that the following theorem is just a consequence of
the Jacobis's functional equation. See \cite[p.2]{chung2}  for an
alternative proof arrived by Chung 1976 and \cite[Theorem
2.6]{kaigh}.
\begin{Theorem}\label{mythrm} For every $x>0$, $f(x)=K(x)$.
\end{Theorem}
\begin{proof} From Theorem \ref{kappa2}, we have
$$f(x)=\lim_{N\rightarrow \infty}\mathbb{P}_{N}(\max_{0\leq i\leq n}
S_{i}\leq x\sqrt{N} )= \sqrt{2\pi
}x^{-3}\sum_{s=1}^{\infty}s^2\pi^2\exp\left(-\frac{s^2\pi^2}{2x^2}\right).$$
 Alternatively, this result is proved below using the Poisson
summation formula suggested by Dr. Rouault \cite{alain}.\\
\textit{Alternative proof of Theorem} \ref{mythrm}. For an
appropriate function $f$, the Poisson summation formula may be
stated as:
$$\sum_{n=-\infty}^{\infty}f(t+nT)=\frac{1}{T}\sum_{k=-\infty}^{\infty}\hat{f}\left(\frac{k}{T}\right)\exp\left(2\pi
i \frac{k}{T}  t\right),$$ where $\hat{f}$ is the Fourier transform
of $f$.  Let $f(x)=\exp(-2x^2)$, so
$\hat{f}=\sqrt{\frac{\pi}{2}}\exp{\left(-\frac{k^2\pi^2}{2}\right)}.$
Therefore, by Poisson summation formula  with $t=0$ and $T=x$ we
have,
$$\sum_{n=-\infty}^{\infty}f(nx)=\frac{1}{x}\sum_{k=-\infty}^{\infty}\hat{f}\left(\frac{k}{x}\right).$$
i.e.,
$$\sum_{n=-\infty}^{\infty}\exp(-2n^2x^2)=\frac{1}{x}\sum_{k=-\infty}^{\infty}\sqrt{\frac{\pi}{2}}\exp\left(\frac{-\pi^2 k^2}{2x^2}\right).$$
Separating positive and negative indices and multiplying by $x$
yields
$$x\left(1+2\sum_{n=1}^{\infty}\exp(-2n^2x^2)\right)=1+2\sum_{k=1}^{\infty}\sqrt{\frac{\pi}{2}}\exp\left(\frac{-\pi^2 k^2}{2x^2}\right).$$
Now by taking the derivative with respect to $x$ we get,
$$1-2\sum_{n=1}^{\infty}\exp(-2n^2x^2)(4n^2x^2-1)=\sqrt{2\pi}\frac{\pi^2k^2}{x^3}\sum_{k=1}^{\infty}\exp\left(\frac{-\pi^2k^2}{2x^2}\right),$$
which implies $f(x)=K(x)$ and completes our proof.
\end{proof}
\section{ Large and  moderate deviation principle for the maximum of  Dyck paths}
Large deviation theory deals with the decay of the probability of
increasingly unlikely events. Historically, the oldest and most
important result  is that the empirical mean of an i.i.d.  sequence
of real-valued random variables obeys a large deviation principle
with rate $n$.
\begin{Theorem}[Cram\'{e}r's Theorem]\label{cramer}
Let $X_1,X_2,\cdots$ be i.i.d. random variables with mean $\mu$
satisfying $\varphi(\lambda):=\log(\mathbb{E}e^{\lambda X})<\infty,$
and let $S_n:=\sum_{i=1}^{n}X_i.$ For any $x>\mu$ we have,
$$\lim_{n\rightarrow \infty}\frac{1}{n}\log
\mathbb{P}\left\{\frac{1}{n}S_n\geq x \right\}=-\varphi^{*}(x),$$
where $\varphi^{*}(x)$, given by $\varphi^{*}(x):=\sup_{\lambda \in
\mathbb{R}}\left\{\lambda x - \varphi(\lambda)\right\}$ is the
Legendre transform of $\varphi$.
\end{Theorem}
The moderate deviation principle gives us the probability of
deviation of $S_n$ of size $a_n$ where $a_n\ll n$, i.e.,
$\frac{a_n}{n}=o(1)$.
\begin{Theorem}[Moderate deviation principle for a random walk]\label{moderate}
Under the same assumptions as in Theorem \ref{cramer}, if
$\sqrt{n}\ll a_n\ll n$ we have, for all $x>0$,
$$\lim_{n\rightarrow \infty}\frac{n}{a_{n}^2}\log
\mathbb{P}\left\{S_n-\mu n \geq x
a_n\right\}=-\frac{x^2}{2\sigma^2}. $$
\end{Theorem}
In this section, we discuss the large and moderate deviation
principle for the maximum of the Dyck path in two cases. We start
with the following lemma which we will use  later.
\begin{lemma}\label{kappagreat2}
For a fixed $n$ and   $N\gg n^2$. We have  for any  $1<s<n$,
$$\frac{G_{N,n}(s)}{G_{N,n}(n)}=o(1) \ \textrm{as} \ N \ \rightarrow
\infty,$$and even $G_{N,n}(1)\gg\sum_{s=2}^{n-1}G_{N,n}(s).$
Therefore,
$$\sum_{s=1}^{n}G_{N,n}(s)=2G_{N,n}(1)+o(1).$$
\end{lemma}
We omit the details of the proof, but it follows from Lemma
\ref{cos} and by an application of the Dominated Convergence
Theorem.
\begin{remark}
It also follows from the Dominated Convergence Theorem that,

$$\lim_{n\rightarrow \infty} \frac{|D_{2N,n}|}{C_{N}}=
\begin{cases}
                              0, & \hbox{$N \gg n^2$;} \\
                              1, & \hbox{$n\ll N\ll  n^2$.}
\end{cases}$$
\end{remark}
We obtain the following large and moderate deviation principle for a
random Dyck path when $N\gg n^2$. Note that $n$ might be any
function of $N$ as long as $N\gg n^2$.
 \begin{Theorem} \label{mod1} Let $N$ and $n$  be  positive integers satisfying
 $N\gg n^2$. Then
$$\lim_{N\rightarrow
\infty}\frac{(n+1)^2}{N}\log \mathbb{P}_{N}( \textrm{max height of
the Dyck Path with length }\ 2N < n) \longrightarrow-\pi^2.$$
\end{Theorem}
\begin{proof}
For sufficiently large $N$ we have,
\begin{eqnarray*}
\frac{|D_{2N,n}|}{C_{N}}&=&\frac{\sum_{s=1}^{n}\frac{2}{n+1}\sin^{2}\left(\frac{\pi
s}{n+1}\right)4^{N}\cos^{2N}\left(\frac{\pi s }{n+1}\right)}
{\frac{4^N}{\sqrt{\pi}N^{\frac{3}{2}}}}\\
&=&\frac{2}{(n+1)}\sqrt{\pi}N^{\frac{3}{2}}\sum_{s=1}^{n}G_{N,n}(s),
\end{eqnarray*}
From Lemma \ref{kappagreat2}, we know that for  sufficiently large
$N$ we have\\
\begin{eqnarray*}
\frac{2}{(n+1)}\sqrt{\pi}N^{\frac{3}{2}}\sum_{s=1}^{n}G_{N,n}(s)&=&
\frac{2}{(n+1)}\sqrt{\pi}N^{\frac{3}{2}}(2G_{N,n}(1)+o(1))\\
&=& \frac{4\sqrt{\pi}N^{\frac{3}{2}}}{n+1}\sin^{2}\left(\frac{\pi
}{n+1}\right)\cos^{2N}\left(\frac{\pi  }{n+1}\right)(1+o(1))\\
&=&\frac{4\sqrt{\pi}N^{\frac{3}{2}}}{n+1}\left(\frac{\pi
}{n+1}\right)^{2}\exp{\left(\frac{-2N\pi^{2}}{2(n+1)^2}\right)}(1+o(1)).
\end{eqnarray*}
Thus,
$$\frac{|D_{2N,n}|}{C_N}=\frac{4\sqrt{\pi}N^{\frac{3}{2}}}{n+1}\left(\frac{\pi
}{n+1}\right)^{2}\exp{\left(\frac{-N\pi^{2}}{(n+1)^2}\right)}(1+o(1)).$$
Therefore,
\begin{eqnarray}
\log\left(\frac{|D_{2N,n}|}{C_N}\right)=\log\left(\frac{4\pi^{\frac{3}{2}}N^{\frac{3}{2}}}{(n+1)^3}\right)-\frac{\pi^2
N}{(n+1)^2}+o(1).
\end{eqnarray}
Multiplying both sides of (2)  by $\frac{(n+1)^2}{N}$ and letting
$N\rightarrow \infty$, we obtain
$$\lim_{N\rightarrow \infty}\frac{(n+1)^2}{N}\log
\mathbb{P}_{N}( \textrm{max of a Dyck Path with length}\ 2N<
n)=-\pi^2.$$
\end{proof}
Finally, we consider the last possible regime for  $n\leq N\ll
n^2$.
\begin{Theorem}\label{mod2} Let $N$ be any positive integer and $x>0$,\\
\begin{itemize}
  \item {If $n\ll N\ll  n^2$, then we
have
$$\lim_{N\rightarrow \infty}\frac{N}{2n^{2}}\log \mathbb{P}_{N}
\left( \textrm{max height of a Dyck path with length}\  2N> x
n\right)\rightarrow -x^2.$$}
  \item {If $n\sim 2N$, then we have
  $$\lim_{N\rightarrow \infty}\frac{1}{2N}\log \mathbb{P}_{N} \left( \textrm{max height of
a Dyck path with length}\  2N> x n\right)\rightarrow h(x),$$ for
$0<x\leq{\frac{1}{2}}^-$ where
$$h(x)=-(x+\frac{1}{2})\log(1+2x)-(\frac{1}{2}-x)\log(1-2x). \\
         $$
 Otherwise,
$$\lim_{N\rightarrow \infty}\frac{1}{2N}\log \mathbb{P}_{N}
\left( \textrm{max height of a Dyck path with length}\  2N> x
n\right)\rightarrow -\infty.$$}
\end{itemize}
\end{Theorem}
\begin{proof}
Let $n\ll N\ll  n^2$. We start by finding a lower bound  for the
number of Dyck paths which hit $xn$ at least once in $2N$ steps.
Consider Dyck paths that hit  a maximum of $xn$ at the $Nth$ step.
We know from \cite{feller1}  that the number of all paths from the
origin to the point $(N,xn)$ is given by $\left(\begin{array}{c}
                                                         N \\
                                                         \frac{N+xn}{2}
                                                       \end{array}\right).$
We have to subtract the number of paths which at some point go below
the  $x$-axis. By  the reflection principle this is
$\left(\begin{array}{c}
                     N \\
                      \frac{N+xn}{2}+1
                    \end{array}\right).$
Thus  the number of Dyck paths with this property in $N$ steps is
$$\left(\begin{array}{c}
                                                         N\\
                                                         \frac{N+xn}{2}
                                                       \end{array}\right)-\left(\begin{array}{c}
                                                         N \\
                                                         \frac{N+xn}{2}+1
                                                       \end{array}\right)=\left(\begin{array}{c}
                                                      N \\
                                                      \frac{N+xn}{2}
                                                    \end{array}\right)\frac{xn+1}{\frac{N+xn}{2}+1}.$$
Therefore, the number of Dyck paths in $2N$ steps with this property
is equal to $\left(\begin{array}{c}
                                                      N \\
                                                      \frac{N+xn}{2}
                                                    \end{array}\right)^{2}\left(\frac{xn+1}{\frac{N+xn}{2}+1}\right)^{2}.$
Thus,
$$\mathbb{P}_{N} \left(\textrm{max height of a Dyck path with length}\  2N>
xn\right)>\frac{\left(\begin{array}{c}
                                                      N \\
                                                      \frac{N+xn}{2}
                                                    \end{array}\right)^{2}\left(\frac{xn+1}{\frac{N+xn}{2}+1}\right)^{2}}{2^{2N}}.$$
It is easy to show that
$$\log\left(\begin{array}{c}
                                                     N \\
                                                      \frac{N+xn}{2}
                                                    \end{array}\right)=N\log2-\frac{(xn)^2}{N}+O\left(\frac{1}{N}\right).$$
Thus
$$\log\left(\begin{array}{c}
                                                      N \\
                                                      \frac{N+xn}{2}
                                                    \end{array}\right)^{2}\left(\frac{xn+1}{\frac{N+xn}{2}+1}\right)^{2}=\log2^{2N}-\frac{2x^2n^2}{N}+\log\left(\frac{xn+1}{\frac{N+xn}{2}+1}\right)^2+O\left(\frac{1}{N}\right).$$
Therefore, $$ \liminf_{N\rightarrow \infty}\frac{N}{2n^{2}}\log
\mathbb{P}_{N} \left( \textrm{max height of a Dyck path with
length}\ 2N> xn\right)\geq
 -x^2.
$$
To determine an upper bound, we consider all
 paths  that  start at the origin, hit $x n$ at
least once and  can end at any point. Therefore, we need to take
$N+xn$ steps out of $2N$ steps upwards. By the reflection principle
the number of such paths is equal to $2\left(
                                           \begin{array}{c}
                                             2N \\
                                            N+x n\\
                                           \end{array}
                                          \right)$.
We thereby  obtain, $$ \mathbb{P}_{N} \left( \textrm{max height of a
Dyck path with length}\ 2N> xn\right)<\frac{2\left(
                                          \begin{array}{c}
                                            2N \\
                                            N+x n                           \\
                                          \end{array}
                                         \right)}{2^{2N}}.$$\\
We can show that $$\log\left(
                                           \begin{array}{c}
                                             2N \\
                                             N+x n \\
                                           \end{array}
                                          \right)=\frac{-2x^2n^2}{N}+2N\log(2)+O\left(\frac{1}{N}\right).$$
Hence,
$$\limsup_{N\rightarrow \infty}\frac{N}{2n^{2}}\log \mathbb{P}_{N} \left(
\textrm{max height of a Dyck path with length}\  2N> x n \right)\leq
-x^2.$$ Now if we assume $n\sim 2N$, then if $0<x<\frac{1}{2}$ we
deduce
$$\mathbb{P}_{N}\left( \textrm{max height of the Dyck path with length}\ 2N> x
2N\right)=\frac{\left(
                                          \begin{array}{c}
                                            2N \\
                                            N+x 2N                           \\
                                          \end{array}
                                         \right)}{2^{2N}},$$
via same argument. Therefore for $0<x<\frac{1}{2}$, we have
$$\log\left(
                                          \begin{array}{c}
                                            2N \\
                                            N+x2N                          \\
                                          \end{array}
                                         \right)=2N\log2-2N(x+\frac{1}{2})\log(1+2x)-2N(\frac{1}{2}-x)\log(1-2x).$$
So we get
$$ \lim_{N\rightarrow \infty}\frac{1}{2N}\log\left(
\frac{\left(
                                          \begin{array}{c}
                                            2N \\
                                            N+x2N                          \\
                                          \end{array}
                                         \right)}{2^{2N}}\right)
                                         =-(x+\frac{1}{2})\log(1+2x)-(\frac{1}{2}-x)\log(1-2x).$$
For the case where $x\geq\frac{1}{2}$, we have
        $$\mathbb{P}_{N} \left(\textrm{max height of the Dyck path with length}\  2N =N\right)
    =\frac{1}{C_{N}},$$
since the  maximum of a Dyck path cannot be greater than half of the
steps and in the case of equality, there is only one path with this
property. So
\begin{eqnarray*}
\lim_{N\rightarrow \infty}\frac{1}{2N}\log
\left(\frac{1}{C_N}\right)&=&\lim_{N\rightarrow
\infty}\frac{1}{2N}\left(\log \left(\frac{1}{4^{N}}\right)+o(1)\right)\\
&=& -\log2.
\end{eqnarray*}
Note
$\lim_{x\rightarrow\frac{1}{2}}(-(x+\frac{1}{2})\log(1+2x)-(\frac{1}{2}-x)\log(1-2x))=-\log2.$
\end{proof}
\begin{remark}
In Theorem \ref{mod2} if we assume $x=1$, we have for $n\ll N\ll
n^2$
$$\lim_{N\rightarrow \infty}\frac{N}{n^{2}}\log \mathbb{P}_{N}
\left( \textrm{max height of a Dyck path with length}\  2N>
n\right)\rightarrow -2.$$
\end{remark}

 \textbf{Acknowledgments.} This paper is supported by
University of Ottawa, and NSERC Canada Discovery Grants of  Dr.
David Handelman. The author wishes to thank Dr. Beno\^{i}t Collins,
for his help, advice and encouragement on this paper. She also
acknowledges vital help from Dr. David McDonald for the algebraic
proof of Theorem \ref{mythrm}.

\end{document}